\documentclass[11pt,a4paper]{amsart}
\usepackage[T1]{fontenc}
\usepackage{lmodern}
\usepackage{amsmath,amssymb,amsthm,mathtools}
\usepackage[margin=30mm]{geometry}
\usepackage{microtype}
\usepackage[colorlinks=true,linkcolor=blue,citecolor=blue,urlcolor=blue]{hyperref}
\usepackage{cleveref}

\newtheorem{theorem}{Theorem}[section]
\newtheorem{lemma}[theorem]{Lemma}
\newtheorem{proposition}[theorem]{Proposition}

\theoremstyle{remark}

\numberwithin{equation}{section}
\newcommand{\A}{\mathcal A}
\newcommand{\F}{\mathcal F}
\newcommand{\T}{\mathcal T}
\newcommand{\St}{\mathcal S}

\begin{document}
\title[Intersecting loose spanning hypertrees]{Intersecting families of loose spanning hypertrees}
\author{Yury Person}
\address{Technical University of Ilmenau, Institute of Mathematics, Weimarer Stra{\ss}e 25, 98693 Ilmenau, Germany}
\email{yury.person@tu-ilmenau.de}
\author{Thomas Schweser}
\address{Technische Hochschule Rosenheim, Fakult\"at f\"ur Angewandte Natur- und Geisteswissenschaften, Hochschulstra{\ss}e 1, 83024 Rosenheim, Germany}
\email{thomas.schweser@th-rosenheim.de}
\hypersetup{pdftitle={Intersecting families of loose spanning hypertrees},pdfauthor={Yury Person and Thomas Schweser}}

\begin{abstract}
A loose hypertree is a hypergraph whose incidence graph is a tree.
Let $b\ge3$, $n\equiv1\pmod{b-1}$, and $s:=(n-1)/(b-1)$.
We determine the largest $t$-intersecting families of loose spanning
$b$-uniform hypertrees on $[n]$, simultaneously for every $1\le t\le s$,
provided $n$ exceeds an explicit constant depending only on $b$.
The extremal families consist of all such hypertrees containing a fixed
$t$-edge loose forest whose component sizes are as equal as the
divisibility condition permits. In particular, when $bt\le n$, the
extremal forests are matchings.
We also prove a Hilton--Milner-type bound: if $n\ge C_b(t+1)$, where
$C_b$ depends only on $b$, every nontrivial $t$-intersecting family has
at most $(s+1)r_{t+1}(n)$ members. Here $r_j(n)$ is the maximum possible  number
of loose spanning $b$-uniform hypertrees containing a fixed $j$-edge
loose forest. For fixed constants $b,t$, this bound is asymptotically
sharp and is $O_{b,t}(n^{2-b})r_t(n)$.
The proofs combine exact forest enumeration with the intersection
theorems of Kupavskii for spread families.
\end{abstract}
\maketitle

\section{Introduction}\label{sec:intro}

A family of sets is \emph{$t$-intersecting} if any two of its members
have at least $t$ common elements. The Erd\H{o}s--Ko--Rado theorem
\cite{ErdosKoRado1961} states that an intersecting family of $k$-subsets
of an $n$-element set has at most $\binom{n-1}{k-1}$ members when
$n\ge2k$. For $t$-intersection, the corresponding bound is
$\binom{n-t}{k-t}$ when $n\ge(t+1)(k-t+1)$, by Wilson's theorem
\cite{Wilson1984EKR}, following work of Frankl \cite{Frankl1978EKR}.
Ahlswede and Khachatrian \cite{AhlswedeKhachatrian1997Complete}
determined the maximum size in all parameter ranges.

A $t$-intersecting family is \emph{trivial} if some $t$-element set is
contained in every member, and \emph{nontrivial} otherwise.
Hilton and Milner \cite{HiltonMilner1967} determined the largest
nontrivial intersecting families of $k$-sets: for $k\ge2$ and $n>2k$,
their maximum size is
\[
 \binom{n-1}{k-1}-\binom{n-k-1}{k-1}+1.
\]
An extremal family consists of a fixed $k$-set $B$ together with
all $k$-sets containing a fixed element $x\notin B$ and meeting $B$.
For $k\ge4$, these are the only extremal families up to relabelling.
For $k=3$, the triples containing at least two elements of a fixed
three-element set also attain the bound; for $k=2$, the extremal
family is a triangle.

For $t$-intersection, Frankl \cite{Frankl1978Nontrivial} determined
the largest nontrivial families for sufficiently large $n$.
Two constructions arise: 
Given $X\subset M\subseteq[n] = \{1, 2, \ldots, n\}$ with
$|X|=t$ and $|M|=k+1$, let $\mathcal H_t$ consist of all $k$-sets
containing $X$ and meeting $M\setminus X$, together with the
$t$ sets $M\setminus\{x\}$, $x\in X$.
Given a $(t+2)$-set $Z$, let $\mathcal A_t$ consist of all $k$-sets
containing at least $t+1$ elements of $Z$.
Ahlswede and Khachatrian
\cite{AhlswedeKhachatrian1996Nontrivial} showed that, for
$1\le t<k$ and $n>(t+1)(k-t+1)$, the largest nontrivial
$t$-intersecting family has size
\begin{equation*}
 \max\left\{
 \binom{n-t}{k-t}-\binom{n-k-1}{k-t}+t,\quad
 (t+2)\binom{n-t-2}{k-t-1}+\binom{n-t-2}{k-t-2}
 \right\},
\end{equation*}
the two terms being $|\mathcal H_t|$ and $|\mathcal A_t|$, respectively.
Here a binomial coefficient with negative lower argument is zero.
Thus one of these two constructions attains the maximum.
The complete nontrivial-intersection theorem from~\cite{AhlswedeKhachatrian1996Nontrivial} also treats the
remaining parameter ranges.

For further background, including
intersection questions for other combinatorial structures, we refer
to the survey of Frankl and Tokushige \cite{FranklTokushige2016Survey}.

In the spanning-tree setting, intersection is measured in edges.
Frankl, Hurlbert, Ihringer, Kupavskii, Lindzey, Meagher and Pantangi~\cite{Frankl2025SpanningTrees} proved that, for $n\ge2^{19}$ and
$2\le t\le n/(4032\log_2 n)$, a $t$-intersecting family of spanning
trees of $K_n$ has at most $2^t n^{n-t-2}$ members, with equality
precisely for the families of all trees containing a fixed matching of $t$ edges.
For $t=1$ and $n\ge2^{19}$, the maximum is instead $2n^{n-3}+n-2$: the extremal family
consists of all trees containing a fixed edge together with all $n$
stars. Saengrungkongka \cite{Saengrungkongka2025tIntersectingSpanningTrees}
obtained the sharp bound $2^t n^{n-t-2}$ throughout the larger range
$2\le t\le cn$, for an absolute constant $c>0$.
Iarovikova and Kupavskii
\cite{IarovikovaKupavskii2025CompleteSpanningTrees} settled every
$2\le t\le n-2$ for sufficiently large $n$: each extremal family
consists of all trees containing a fixed $t$-edge forest whose component sizes differ by at
most one. These forests are matchings when $2t\le n$.
For $t=n-1$, a family contains at most one tree.

Bruns, Desai, Gavrilyuk, Gomez and Lindzey
\cite{Bruns2026BipartiteSpanningTrees} considered spanning trees of
$K_{n,n}$. For sufficiently large $n$, they determined the intersecting
case exactly and, for $2\le t\le n/(C\log_2 n)$, showed that every extremal
family consists of a full matching star together with a lower-order family of
exceptional trees. Thus, exceptional trees can occur even when $t\ge2$.
These results motivate the corresponding questions for loose
spanning uniform hypertrees.

\subsection{Loose hypertrees and the main results}
A hypergraph is \emph{$b$-uniform} if every hyperedge has $b$ vertices,
and \emph{linear} if any two distinct hyperedges intersect in at most
one vertex. A \emph{matching} is a set of pairwise disjoint hyperedges.
A \emph{Berge cycle} of length $k\ge2$ consists of distinct vertices
$v_1,\ldots,v_k$ and distinct hyperedges $E_1,\ldots,E_k$ such that
$\{v_i,v_{i+1}\}\subseteq E_i$ for every $i$, with $v_{k+1}=v_1$.
The \emph{incidence graph} of a hypergraph is the bipartite graph whose
two vertex classes are its vertices and its hyperedges, with adjacency
recording containment. A \emph{loose forest} is a hypergraph whose
incidence graph is a forest, and a \emph{loose hypertree} is one whose
incidence graph is a tree. Equivalently, a loose forest has no Berge
cycle; in particular, it is linear. We include isolated vertices when
counting components. From now on, forests and hypertrees are understood
to be loose; we usually omit this adjective outside theorem statements.

Write $\T_n^{(b)}$ for all loose spanning $b$-uniform hypertrees on
$[n]$. It follows from the definition of a loose hypertree that every such $b$-uniform hypertree on $n$ vertices  has
exactly
\begin{equation}\label{eq:edge_count}
 s:=s(n,b):=\frac{n-1}{b-1}
\end{equation}
hyperedges. We use $s$ with this meaning throughout the paper.
Unless explicitly stated otherwise, $b\ge3$ and $n\equiv1\pmod{b-1}$. For a vertex $v$ of a hypertree $T$ we write $\deg_T(v)$ for the number of edges it is contained in. A loose hypertree $T$ on $n$ vertices satisfies
\[
\sum_{v\in V(T)}\deg_T(v)=n+s-1.
\]

We identify a hypertree or forest with its set of hyperedges, so
$\T_n^{(b)}\subseteq\binom{\binom{[n]}b}{s}$.
For a forest $F$ on $[n]$, let
\[
 \T_n^{(b)}[F]=\{T\in\T_n^{(b)}:F\subseteq T\},\qquad
 T_b(F)=|\T_n^{(b)}[F]|.
\]
The family $\T_n^{(b)}[F]$ is the \emph{star at $F$}; it is a
\emph{$j$-edge star} when $|F|=j$. For $0\le j\le s$, define
\[
 r_j(n)=\max\{T_b(F):F\text{ is a $b$-uniform $j$-edge forest on }[n]\}.
\]
In particular, $r_0(n)=|\T_n^{(b)}|$.

Our first result is a complete intersection theorem: its threshold
depends only on $b$, rather than on $t$. Set
\begin{equation}\label{eq:ekr_threshold}
 B_b=b^{2b}(b-2)!,\qquad
 N_b=\left\lceil
       \left(\frac{48B_b}{b-1}\right)^{1/(b-2)}
      \right\rceil.
\end{equation}

\begin{theorem}\label{thm:ekr}
Let $b\ge3$, $n\equiv1\pmod{b-1}$, $n\ge N_b$, and
$1\le t\le s=(n-1)/(b-1)$. Every $t$-intersecting family
$\F\subseteq\T_n^{(b)}$ satisfies $|\F|\le r_t(n)$.
Equality holds if and only if $\F=\T_n^{(b)}[F]$ for a loose
$b$-uniform $t$-edge forest $F$ whose component sizes differ by at most $b-1$.
In particular, if $bt\le n$, then
\begin{equation}\label{eq:matching_formula}
 r_t(n)=
 \frac{(n-t(b-1)-1)!}{(s-t)!\,((b-1)!)^{s-t}}
 b^t n^{s-t-1},
\end{equation}
and equality holds precisely for stars at matchings of $t$ hyperedges.
\end{theorem}

Thus the matching-star statement holds with the explicit bound
$n_0(b,t)\le\max\{N_b,bt\}$, while the general statement holds
simultaneously for every meaningful $t$. The condition $bt\le n$
is necessary for a matching of $t$ hyperedges to exist.
In contrast with the graph case~\cite{Frankl2025SpanningTrees}, $t=1$ has no exceptional extremal family.

For our nontrivial-family bound, define
\begin{equation}\label{eq:hm_threshold}
 C_b=2^{21}B_b\log_2(2B_b).
\end{equation}

\begin{theorem}\label{thm:hm}
Let $b\ge3$, $n\equiv1\pmod{b-1}$, $t\ge1$, and $n\ge C_b(t+1)$.
Every nontrivial $t$-intersecting family $\F\subseteq\T_n^{(b)}$ of
loose spanning $b$-uniform hypertrees satisfies
\begin{equation}\label{eq:hm_bound}
 |\F|\le(s+1)r_{t+1}(n).
\end{equation}
For fixed constants $b$ and $t$, the right-hand side is
$O_{b,t}(n^{2-b})r_t(n)$, and there are nontrivial $t$-intersecting
families of size $(1-o(1))s r_{t+1}(n)$ as $n\to\infty$ through
admissible values.
\end{theorem}

Theorem~\ref{thm:hm} is a Hilton--Milner-type bound, rather than an
exact finite-$n$ classification. Its leading term is nevertheless
best possible for fixed $b,t$.

\subsection{Proof strategy}
For a given number of edges, our enumeration formula reduces
the maximization of star sizes to maximizing the product of
component sizes. We use this formula in two ways.
First, the maximum star sizes satisfy
$r_{j+1}(n)\le B_b n^{1-b}r_j(n)$ at every depth.
This permits a direct application of Kupavskii's theorem for
nontrivial intersecting families \cite[Theorem 14]{Kupavskii2023Partitions} to the original
family and proves Theorem~\ref{thm:ekr}.
Second, the same estimate gives global ambient spread. Together with
Kupavskii's spread approximation theorem
\cite[Theorem 13]{Kupavskii2023Partitions}, it gives the sharper
bound in Theorem~\ref{thm:hm}.
The spread approximation method originates in the sunflower
breakthrough of Alweiss, Lovett, Wu and Zhang
\cite{AlweissLovettWuZhang2021Sunflower} and the work of Kupavskii
and Zakharov \cite{KupavskiiZakharov2024ForbiddenIntersections}.

Section~\ref{sec:forests} gives the enumeration and degree estimates.
Section~\ref{sec:methods} states the intersection tools and checks
their parameters. Sections~\ref{sec:ekr} and~\ref{sec:hm} prove
the two main theorems, including the asymptotically sharp construction.
Section~\ref{sec:remarks} discusses the parameter ranges and the graph case.

\subsection*{Statement on AI use}
We first obtained the proof of Theorem~\ref{thm:ekr} with larger bounds on $n$, where we used OpenAI Codex to only assist with literature search and to provide Lemma~\ref{lem:enumeration}. Later we used OpenAI Codex and GPT-6 Astra to obtain Hilton-Milner-type bound in Theorem~\ref{thm:hm} and to obtain better bounds from which both theorems hold. We also used assistance with writing the paper. After that we reviewed  and
edited the content as needed and  verified all proofs.

\section{Loose forests and their extensions}\label{sec:forests}

Throughout the rest of the paper, put $d:=b-1$ and retain
$s=(n-1)/d$ from \eqref{eq:edge_count}, the number of edges in a
spanning hypertree. If a $b$-uniform forest on $[n]$ has $j$ edges and $m$
components, its incidence graph has $n+j$ vertices and $bj$ edges.
Since the incidence graph is a forest with $m$ components, its number
of vertices minus its number of edges equals $m$. Thus,
\begin{equation}\label{eq:components}
 m=n-dj.
\end{equation}
In particular, each component with $e$ edges has $1+de$ vertices.
A \emph{loose path} is an ordered sequence of hyperedges in which
consecutive edges meet in one vertex and nonconsecutive edges are
disjoint. A path with $e$ edges has $1+de$ vertices.

\subsection{Enumeration}
A classical enumeration formula of Selivanov~\cite{Selivanov1972HomogeneousHypergraphs}
(see also Lavault~\cite[Corollary 1]{Lavault2011PruferHypertrees},
after dividing by the number of possible roots) gives
\begin{equation}\label{eq:total}
 |\T_n^{(b)}|=
 \frac{(n-1)!}{s!\,(d!)^s}n^{s-1}
 =\frac{n!}{s!\,(d!)^s}n^{s-2},
\end{equation}
which generalizes Cayley's theorem about the number of (labelled) spanning trees in complete graphs. 
Bacher \cite{Bacher2011HypertreeEnumeration} used Pr\"ufer codes to generalize Selivanov's formula to
hypertrees with prescribed edge sizes and vertex degrees. The following
is a uniform specialization of his enumeration formula
\cite[Theorem 1.1]{Bacher2011HypertreeEnumeration}. In Bacher's notation,
it follows by setting $n=m-1$, $k=u$,
$\lambda_1=\cdots=\lambda_u=d$, and $\mu_i=a_{i+1}-1$ for
$0\le i\le m-1$; the multiplicities of the parts are then $\nu_d=u$
and $\nu_j=0$ for $j\ne d$.

\begin{proposition}[Uniform degree enumeration]\label{prop:bacher}
Let $b\ge3$, $d=b-1$, $u\ge1$, and $m=1+du$. For positive integers
$a_1,\ldots,a_m$ satisfying $\sum_{i=1}^m(a_i-1)=u-1$, the number
of loose $b$-uniform hypertrees on $[m]$ in which vertex $i$ has
degree $a_i$ is
\begin{equation}\label{eq:bacher_degrees}
 \frac{d(m-2)!}{(d!)^u\prod_{i=1}^m(a_i-1)!}.
\end{equation}
In particular, every such degree sequence is realizable.
\end{proposition}

The following lemma is the main enumerative ingredient of our proofs.
It extends the forest-extension formula of Lu, Mohr and
Sz\'ekely~\cite[Lemma 6]{LuMohrSzekely2014Quest} to uniform loose hypertrees
and allows us to determine the maximum star sizes and derive the spread
estimates needed for both main theorems. The proof follows their
contraction argument, using Bacher's degree enumeration formula.

\begin{lemma}\label{lem:enumeration}
Let $b\ge3$, $n\equiv1\pmod{b-1}$, and let $F$ be a loose
$b$-uniform $j$-edge forest on $[n]$, with component sizes
$f_1,\ldots,f_m$. Put $u=s-j=(m-1)/d$, i.e., $u$ is the number of edges that
must be added to complete $F$ to a spanning hypertree. Then
\begin{equation}\label{eq:forest_formula}
 T_b(F)=\frac{(m-1)!}{u!\,(d!)^u}
         \left(\prod_{i=1}^m f_i\right)n^{u-1}.
\end{equation}
\end{lemma}

\begin{proof}
If $u=0$, then $F$ is already a spanning hypertree, so $T_b(F)=1$.
Moreover, $m=1$ and $f_1=n$, so the right-hand side also equals one.
Suppose $u\ge1$. Contract the components of $F$ to $m$ labelled
vertices. Every added edge must meet $b$ distinct components;
otherwise it creates a cycle in the incidence graph. A completion
of $F$ therefore gives a $b$-uniform hypertree with $u$ edges on the
contracted vertices. Conversely, such a hypertree, together with a
choice of an attachment vertex in each incident component, gives a
unique completion: expanding its vertices back into the components
of $F$ preserves connectedness and creates no incidence cycle.

The contracted degrees $a_1,\ldots,a_m$ are positive and satisfy
$\sum_i(a_i-1)=u-1$. Conversely, every such sequence is realizable
by Proposition~\ref{prop:bacher}; thus we may sum over all positive degree
sequences satisfying this identity. For each contracted hypertree with
degree sequence $(a_1,\ldots,a_m)$, there are $\prod_i f_i^{a_i}$
independent choices of attachment vertices. Setting $z_i=a_i-1$ and
applying the multinomial theorem gives
\begin{align*}
 T_b(F)
 &=\frac{d(m-2)!}{(d!)^u}
   \sum_{\substack{a_1,\ldots,a_m\ge1\\\sum_i(a_i-1)=u-1}}
       \frac{\prod_i f_i^{a_i}}{\prod_i(a_i-1)!}\\
 &=\frac{d(m-2)!}{(d!)^u}\left(\prod_i f_i\right)
   \sum_{\substack{z_1,\ldots,z_m\ge0\\\sum_i z_i=u-1}}
       \frac{\prod_i f_i^{z_i}}{\prod_i z_i!}\\
 &=\frac{d(m-2)!}{(u-1)!\,(d!)^u}
   \left(\prod_i f_i\right)\left(\sum_i f_i\right)^{u-1}.
\end{align*}
Since $\sum_i f_i=n$ and $du=m-1$, we obtain 
\eqref{eq:forest_formula}.
\end{proof}

We now determine which $j$-edge forests maximize $T_b(F)$.
For fixed $j$, both the number $m=n-dj$ of components and the number
$u=s-j$ of edges still to be added are fixed. By
\eqref{eq:forest_formula}, it remains to maximize $\prod_i f_i$.
Since a component with $e_i$ edges has $1+de_i$ vertices, this
amounts to distributing the $j$ edges among $m$ components
so as to maximize
\[
 \prod_{i=1}^{m}(1+de_i)
 \qquad\text{subject to}\qquad
 \sum_{i=1}^{m}e_i=j,\quad e_i\in\mathbb Z_{\ge0}.
\]
Write $j=am+\ell$,
$0\le\ell<m$: the most even distribution has $\ell$ components
with $a+1$ edges and $m-\ell$ with $a$ edges.

\begin{lemma}\label{lem:maximum}
Let $b\ge3$, $n\equiv1\pmod{b-1}$, and $0\le j\le s$.
Put $m=n-dj$, the number of components, and write
$j=am+\ell$, where $0\le\ell<m$. Then
\begin{equation}\label{eq:general_maximum}
 r_j(n)=\frac{(m-1)!}{(s-j)!\,(d!)^{s-j}}n^{s-j-1}
        (1+da)^{m-\ell}(1+d(a+1))^{\ell}.
\end{equation}
A loose $b$-uniform forest attains this maximum if and only if its component edge
counts differ by at most one, equivalently its component sizes
differ by at most $d$.
If $bj\le n$, the maximizing forests are matchings and their component-size
product is $b^j$.
\end{lemma}

\begin{proof}
Let $e_1,\ldots,e_m\ge0$ be the component edge counts, with sum $j$.
By Lemma~\ref{lem:enumeration}, we need to maximize
$\prod_i(1+de_i)$. If $e_i\ge e_h+2$, replacing these two counts
by $e_i-1,e_h+1$ gives
\[
\begin{aligned}
 &(1+d(e_i-1))(1+d(e_h+1))
  -(1+de_i)(1+de_h)\\
 &\qquad=d^2(e_i-e_h-1)>0.
\end{aligned}
\]
Hence this replacement increases the product, and so
the maximizing counts are $a$ and $a+1$,
with the indicated multiplicities. Every such allocation is
realizable: partition $[n]$ into sets $V_i$ of sizes $1+de_i$, and on
each $V_i$ take a loose path with $e_i$ edges, or an isolated vertex
if $e_i=0$.
Finally, $bj\le n$ is equivalent to $j\le m$; in this case the
balanced counts are zero and one (or all one when $j=m$).
\end{proof}

\subsection{Uniform decay of the maximum star sizes}
The following lemma controls how the maximum star size decreases when
the number of fixed edges is increased by one. It applies regardless of how many edges have
already been fixed. We use it in both main proofs and to show that
nontrivial intersecting families are much smaller than the largest stars.
We do not attempt to optimize the constant $B_b$.

\begin{lemma}\label{lem:degree_decay}
Let $b\ge3$, $n\equiv1\pmod{b-1}$, $d=b-1$, and $s=(n-1)/d$.
For every $0\le j<s$,
\begin{equation}\label{eq:degree_decay}
 r_{j+1}(n)\le\frac{B_b}{n^d}r_j(n),\qquad B_b=b^{2b}(b-2)!.
\end{equation}
\end{lemma}

\begin{proof}
Recall from Lemma~\ref{lem:enumeration} that, among forests with $j$
edges, the star size is maximized precisely when the product of the
component sizes is maximized. Let $P_j$ denote this maximum product,
and put $m=n-dj\ge b$.
By \eqref{eq:components}, increasing the edge count by one decreases
the number of components by $d$. A maximizing $(j+1)$-edge forest
therefore has $m-d$ components. Its largest component has
$f=1+de$ vertices for some $e\ge1$, and
$f\ge n/(m-d)$ because its size is at least the average component size.
Let $Q$ be the product of the sizes of all other components, with
$Q=1$ if there are none. Then $P_{j+1}=Qf$.

We keep the other components unchanged and replace the largest one
by $b$ components on the same vertices, distributing $e-1$ edges
among them as evenly as possible. Their sizes sum to
$b+d(e-1)=1+de=f$, so such a replacement is possible, for example
using paths and isolated vertices. The resulting forest, which
is not necessarily a subforest of the original one, has $j$
edges and $(m-d)-1+b=m$ components.

Write $e-1=ba+r$, $0\le r<b$. The sizes of the new components are $1+da$
and $1+d(a+1)$. To bound them below in terms of $f$, substitute
$e=ba+r+1$ and use $r+1\le b$ and $d=b-1$:
\[
 b^2(1+da)-f
 =b^2-1-d(r+1)+dab(b-1)\ge b-1>0.
\]
Hence every new component has at least $f/b^2$ vertices.
The other component sizes still have product $Q$, so the new
forest has component-size product at least $Q(f/b^2)^b$.
Since $P_j$ is the maximum over all $j$-edge forests, we obtain
\begin{equation}\label{eq:product_comparison}
 \frac{P_{j+1}}{P_j}
 \le\frac{Qf}{Q(f/b^2)^b}
 =\frac{b^{2b}}{f^{b-1}}
 \le b^{2b}\left(\frac{m-d}{n}\right)^d,
\end{equation}
where the last inequality follows from $f\ge n/(m-d)$.
By \eqref{eq:forest_formula},
\[
 r_j(n)=\frac{(m-1)!}{u!\,(d!)^u}n^{u-1}P_j,
 \qquad u=s-j.
\]
We have bounded $P_{j+1}/P_j$ above. It remains to compare the
prefactors. When $j$ increases by one, $m$ and $u$ decrease to
$m-d$ and $u-1$, respectively. Thus the ratio of the prefactors is
\[
 \frac{(m-d-1)!}{(u-1)!\,(d!)^{u-1}}n^{u-2}
 \bigg/
 \frac{(m-1)!}{u!\,(d!)^u}n^{u-1}
 =
 \frac{u\,d!}{n(m-1)(m-2)\cdots(m-d)}
 =
 \frac{(d-1)!}{n\prod_{h=2}^{d}(m-h)},
\]
where the last equality uses $du=m-1$.
Each of the $d-1$ factors $m-h$ is at least $m-d$.
Combining this ratio with \eqref{eq:product_comparison} gives
\[
 \frac{r_{j+1}(n)}{r_j(n)}
 \le\frac{b^{2b}(d-1)!}{n^{d+1}}
       \frac{(m-d)^d}{\prod_{h=2}^d(m-h)}
 \le\frac{B_b(m-d)}{n^{d+1}}
 \le\frac{B_b}{n^d}.
\]
\end{proof}

\subsection{Ratios of matching stars}
We will also use the exact ratios of matching stars. Put
$c_b=b(b-2)!$. Whenever $b(j+1)\le n$, Lemma~\ref{lem:maximum} shows that the
maximum stars at both ranks are matching stars, with component-size
products $b^j$ and $b^{j+1}$, respectively. Taking the ratio of
their sizes in \eqref{eq:forest_formula} gives
\begin{equation}\label{eq:matching_ratio}
 \frac{r_{j+1}(n)}{r_j(n)}
 =\frac{c_b}{n\prod_{h=2}^d(n-dj-h)}
 \ge\frac{c_b}{n^d}.
\end{equation}
Thus, if $bk\le n$,
\begin{equation}\label{eq:matching_lower}
 \frac{r_k(n)}{|\T_n^{(b)}|}\ge
       \left(\frac{c_b}{n^d}\right)^k.
\end{equation}
For fixed $b,j$, the ratio in \eqref{eq:matching_ratio} is
$(c_b+O_{b,j}(n^{-1}))n^{-d}$, where subscripts in \(O\)-notation
indicate the parameters on which the implicit constant may depend.

\section{Intersection methods}\label{sec:methods}

\subsection{Notation and the theorems of Kupavskii}
Kupavskii and Zakharov~\cite{KupavskiiZakharov2024ForbiddenIntersections}
developed the spread approximation technique as a framework for proving
EKR-type results. The idea is to describe most members of an intersecting
family through small subsets they contain, while retaining the
intersection property among these subsets. We use two results of
Kupavskii~\cite{Kupavskii2023Partitions} based on this approach.

Let $\Omega$ be a finite ground set and $\A\subseteq2^\Omega$.
For $X\subseteq\Omega$, the \emph{star of $\A$ at $X$} and the
\emph{link of $\A$ at $X$} are, respectively,
\[
 \A[X]:=\{A\in\A:X\subseteq A\},\qquad
 \A(X):=\{A\setminus X:A\in\A[X]\}.
\]
Thus the star consists of the members containing $X$, and the link
is obtained by deleting $X$ from each of them. This deletion is a
bijection, so $|\A[X]|=|\A(X)|$.
For a family $\St\subseteq2^\Omega$, write
\[
 \A[\St]:=\bigcup_{X\in\St}\A[X]
\]
for the union of the stars indexed by $\St$.
For a nonnegative integer $q$, we use
$\binom\Omega{\le q}:=\bigcup_{i=0}^q\binom\Omega i$, 
where $\binom{\Omega}{i}$ denotes the family
of all $i$-element subsets of $\Omega$.

For a real $R>0$, a nonempty family $\A$ is \emph{$R$-spread} if
$|\A[X]|\le R^{-|X|}|\A|$ for every $X\subseteq\Omega$.
We use Kupavskii's notion of weak spreadness~\cite{Kupavskii2023Partitions},
expressed in terms of maximum degrees as in Wen and
Lv~\cite{WenLv2026UnifiedCrossIntersection}.
For an integer $0\le t\le|\Omega|$, define its \emph{maximum $t$-degree} by
\[
 \Delta_t(\A):=\max_{X\in\binom\Omega t}|\A[X]|.
\]
We say that $\A$ is \emph{weakly $(R,t)$-spread} if
\begin{equation}\label{eq:weak_spread}
 |\A[U]|\le R^{-(|U|-t)}\Delta_t(\A)
 \quad\text{for every }U\subseteq\Omega\text{ with }|U|\ge t.
\end{equation}
We use $t$-intersection for all pairs of
members, including a member paired with itself. Thus every set
in a $t$-intersecting family has at least $t$ elements.

The following two results of Kupavskii allow us to characterize the
maximum $t$-intersecting families of loose spanning hypertrees and
bound the size of nontrivial families. We state them in our terminology,
omitting conclusions that we do not need.

\begin{theorem}[{Kupavskii~\cite[Theorem 13]{Kupavskii2023Partitions}}]\label{thm:spread_approximation}
Let $\Omega$ be a finite set and let $k,t,q$ be positive integers
with $t\le k$ and $q\ge t$.
Let $\A\subseteq\binom\Omega k$ be $R$-spread and let
$\F\subseteq\A$ be $t$-intersecting. Suppose that
\[
 r>2^{12}\log_2(2k),\qquad r\ge2q,\qquad R>r.
\]
There are a $t$-intersecting family
$\St\subseteq\binom\Omega{\le q}$ and $\F'\subseteq\F$ such that
\[
 \F\setminus\F'\subseteq\A[\St],\qquad
 |\F'|\le(r/R)^{q+1}|\A|.
\]
\end{theorem}

\begin{theorem}[{Kupavskii~\cite[Theorem 14]{Kupavskii2023Partitions}}]\label{thm:nontrivial_cores}
Let $\Omega$ be a finite set, let $1\le t\le|\Omega|$ and $q\ge t$
be integers, and let $R\ge1$. Let $\A\subseteq2^\Omega$ be weakly
$(R,t)$-spread. If
$\St\subseteq\binom\Omega{\le q}$ is a nontrivial $t$-intersecting
family and $0<\varepsilon\le1$ satisfies $\varepsilon R\ge24q$,
then
\[
 |\A[\St]|\le\varepsilon \Delta_t(\A).
\]
\end{theorem}

In our application, $\Omega=\binom{[n]}b$, $\A=\T_n^{(b)}$,
$k=s$, and $\Delta_t(\A)=r_t(n)$.
Thus each element of $\Omega$ represents a potential hyperedge,
and each member of $\A$ is the edge set of a spanning hypertree
with $s=(n-1)/(b-1)$ hyperedges.

For $\St \subseteq \binom\Omega{\le q}$, each member $S\in\St$ is a set of at most $q$ hyperedges. If
$\A[S]\ne\varnothing$, then $S$ is contained in a spanning hypertree
and hence is a forest. Members with $\A[S]=\varnothing$ contribute
nothing to the union $\A[\St]$ and may be deleted.

\subsection{Verification of spread and approximation parameters}
Lemma~\ref{lem:degree_decay} has two immediate consequences:
$\T_n^{(b)}$ is $n^d/B_b$-spread, and
\begin{equation}\label{eq:weak_all_depths}
 \T_n^{(b)}\text{ is weakly }
 \left(\frac{n^d}{B_b},t\right)\text{-spread for every }0\le t\le s.
\end{equation}
Indeed, if $U\subseteq\Omega$ is a forest with $|U|=t+h\le s$,
then
\[
 |\T_n^{(b)}[U]|\le r_{t+h}(n)
 \le(B_b/n^d)^h r_t(n)
 =(B_b/n^d)^h\Delta_t(\T_n^{(b)}),
\]
by telescoping \eqref{eq:degree_decay}. If $|U|>s$ or $U$ is not
a forest, then $\T_n^{(b)}[U]=\varnothing$.
Taking $t=0$, for which $r_0(n)=|\T_n^{(b)}|$, also proves global spread.

Recall from \eqref{eq:ekr_threshold} that $B_b=b^{2b}(b-2)!$.
We will apply both intersection theorems with the ambient spread parameter
\begin{equation}\label{eq:hm_spread}
 R=\frac{n^d}{B_b}.
\end{equation}
For Theorem~\ref{thm:hm}, the assumption
$n\ge C_b(t+1)$, where $C_b=2^{21}B_b\log_2(2B_b)$, lets us
choose the approximation parameters so that its remainder has
at most $r_{t+1}(n)$ members. The next lemma verifies this estimate
and bounds the contribution of a nontrivial $t$-intersecting family $\St$.

\begin{lemma}\label{lem:parameters}
Let $t\ge1$ and suppose $n\ge C_b(t+1)$, with $C_b$ as in \eqref{eq:hm_threshold}.
Set
\[
 q=2t+2,\qquad r=2^{13}\log_2(2s)+4t+4,\qquad R=n^d/B_b.
\]
For every $t$-intersecting family $\F\subseteq\T_n^{(b)}$,
Theorem~\ref{thm:spread_approximation} applies with $k=s$ and
$\A=\T_n^{(b)}$, and its remainder satisfies $|\F'|\le r_{t+1}(n)$.
Moreover,
\begin{equation}\label{eq:core_budget}
 0<\frac{24q}{R}\le1,\qquad
 \frac{24q}{R}r_t(n)\le s r_{t+1}(n).
\end{equation}
\end{lemma}

\begin{proof}
Write $B=B_b$, $L=\log_2(2B)$, and $C=C_b=2^{21}BL$.
Here $B\ge3^6$ and so $L>10$. The displayed constants give
$C\ge\max\{b,2d\}$, so $t+1\le s$ and $b(t+1)\le n$.
Consequently, the matching-star estimates
\eqref{eq:matching_ratio} and \eqref{eq:matching_lower} apply
up to rank $t+1$. The ambient family is $R$-spread by
\eqref{eq:hm_spread}, and clearly
$r>2^{12}\log_2(2s)$ and $r\ge2q$.

Since $L>10$ and $\log_2L\le L/2$ for $L\ge4$, we have
\[
 \log_2(2C)=21+L+\log_2L\le4L.
\]
Since $2s\le n$ and $n\ge C(t+1)$, the definition of $r$ gives
\[
 \frac rn
 =2^{13}\frac{\log_2(2s)}{n}+\frac{4(t+1)}{n}
 \le2^{13}\frac{\log_2 n}{n}+\frac4C.
\]
Moreover, $t\ge1$ implies $n\ge2C$. Since $\log_2 x/x$ decreases
for $x\ge3$, we have
\[
 \frac{\log_2 n}{n}\le\frac{\log_2(2C)}{2C}.
\]
Combining these estimates with $\log_2(2C)\le4L$, we obtain
\[
 \frac rn
 \le\frac{2^{12}\log_2(2C)+4}{C}
 \le\frac{2^{14}L+4}{C}
 <\frac{2^{15}L}{C}
 =\frac1{2^6B},
\]
where the strict inequality uses $4<2^{14}L$ and the last equality
uses $C=2^{21}BL$.
In particular, $r<R=n^d/B$, completing the hypothesis check for
Theorem~\ref{thm:spread_approximation}.

Theorem~\ref{thm:spread_approximation} bounds the remainder by
$|\F'|\le(r/R)^{q+1}|\A|$. On the other hand,
\eqref{eq:matching_lower} gives
$|\A|/r_{t+1}(n)\le(n^d/c_b)^{t+1}$, where $c_b=b(b-2)!$.
Since $r/n<1/(2^6B)$ and $R=n^d/B$, we have
\[
 \frac rR=\frac rn\cdot\frac{B}{n^{d-1}}
 <\frac1{2^6n^{d-1}}.
\]
Combining these bounds yields
\begin{align*}
 \frac{|\F'|}{r_{t+1}(n)}
 &\le\left(\frac rR\right)^{q+1}
       \frac{|\A|}{r_{t+1}(n)}\\
 &\le\left(\frac1{2^6n^{d-1}}\right)^{2t+3}
       \left(\frac{n^d}{c_b}\right)^{t+1}\\
 &=2^{-6(2t+3)}c_b^{-(t+1)}n^{-(d-2)t-2d+3}\le1.
\end{align*}
The last inequality uses $c_b\ge1$ and
$-(d-2)t-2d+3\le-1$ for $d\ge2$.

It remains to prove \eqref{eq:core_budget}. By
\eqref{eq:matching_ratio}, we have
$r_{t+1}(n)\ge(c_b/n^d)r_t(n)$. Since $R=n^d/B$, it therefore
suffices for the second inequality to show that
\[
 \frac{24qB}{c_b}\le s.
\]
Our choice of $C$ ensures this: since $c_b\ge d$ and $L>10$,
we have $C=2^{21}BL\ge96dB/c_b$. Using $q=2(t+1)$ and
$t+1\le n/C$, we obtain
\[
 \frac{24qB}{c_b}
 =\frac{48(t+1)B}{c_b}
 \le\frac{48nB}{Cc_b}
 \le\frac{n}{2d}
 \le s,
\]
where the last inequality follows from $s=(n-1)/d$ and $n\ge2$.
Consequently,
\[
 \frac{24q}{R}r_t(n)
 \le s\frac{c_b}{n^d}r_t(n)
 \le s\,r_{t+1}(n).
\]
Finally, $n\ge C\ge B\ge c_b$, $s\le n$, and $d\ge2$ give
\[
 0<\frac{24q}{R}
 \le\frac{s c_b}{n^d}
 \le n^{2-d}
 \le1,
\]
which proves the first inequality in \eqref{eq:core_budget}.
\end{proof}

\subsection{Bounding the intersection with a star}
The following lemma bounds how many members of a nontrivial
$t$-intersecting family can lie in a single $t$-star. We will use it
when the family of cores obtained by spread approximation is trivial.
\begin{lemma}\label{lem:witness}
Let $1\le t<s$, let $X$ be a loose $b$-uniform $t$-edge forest, and let $\F\subseteq\T_n^{(b)}$
be $t$-intersecting. If some $T_0\in\F$ does not contain $X$, then
\[
 |\F\cap\T_n^{(b)}[X]|\le s r_{t+1}(n).
\]
\end{lemma}

\begin{proof}
Every $T\in\F$ containing $X$ must contain an edge $e$ of $T_0\setminus X$,
because $|T_0\cap X|<t$. There are at most $s$ choices for that edge,
and each resulting $(t+1)$-edge set $X\cup\{e\}$ belongs to at most $r_{t+1}(n)$
hypertrees. Infeasible sets contribute zero.
\end{proof}

\section{Proof of Theorem~\ref{thm:ekr}}\label{sec:ekr}

Recall that every spanning hypertree has $s=(n-1)/(b-1)$ edges.
The case $t=s$ is immediate: a family contains at most one hypertree,
and $r_s(n)=1$. Assume $t<s$ and put $\A=\T_n^{(b)}$.
If $\F$ is nontrivial, apply Theorem~\ref{thm:nontrivial_cores}
directly with $\St=\F$, $q=s$, $\varepsilon=1/2$, and
$R=n^d/B_b$. The weak spread condition is \eqref{eq:weak_all_depths},
and $n\ge N_b$ implies
\[
 R\ge\frac{48n}{d}>48s.
\]
Moreover, $\A[\F]=\F$, because both families are $s$-uniform.
We conclude that $|\F|\le r_t(n)/2$.

If $\F$ is trivial, choose a common $t$-edge forest $X$.
Then $|\F|\le T_b(X)\le r_t(n)$. Equality forces
$\F=\T_n^{(b)}[X]$ and $T_b(X)=r_t(n)$.
Lemma~\ref{lem:maximum} gives exactly the stated characterization
and the matching formula.
\qed

\section{Proof of Theorem~\ref{thm:hm}}\label{sec:hm}

\begin{proof}
We first prove the upper bound. Let $\F$ be nontrivial and apply Theorem~\ref{thm:spread_approximation}
with the parameters in Lemma~\ref{lem:parameters}.
We obtain $\St$ and $\F'$ with
$|\F'|\le r_{t+1}(n)$ and
$\F\setminus\F'\subseteq\A[\St]$, where $\A=\T_n^{(b)}$.
Delete every $S\in\St$ for which $\A[S]=\varnothing$, meaning that
no spanning hypertree contains $S$. Such a set contributes nothing
to $\A[\St]$, so the covering inclusion is unchanged. If $\St$
is then empty, $|\F|\le r_{t+1}(n)$ already proves the upper bound.
Assume otherwise that $\St\ne\varnothing$.

If $\St$ is nontrivial, apply Theorem~\ref{thm:nontrivial_cores}
with $\varepsilon=24q/R$. Lemma~\ref{lem:parameters} yields
\[
 |\A[\St]|\le\frac{24q}{R}r_t(n)\le s r_{t+1}(n).
\]
If $\St$ is trivial, choose a common $t$-edge subset $X$ of its
members. It is a forest, and $\A[\St]\subseteq\A[X]$.
Since $\F$ is nontrivial, some $T_0\in\F$ omits an edge of $X$.
Lemma~\ref{lem:witness} gives
$|\F\cap\A[\St]|\le s r_{t+1}(n)$.
In either nonempty case, adding $|\F'|$ proves
\[
 |\F|\le(s+1)r_{t+1}(n).
\]
For fixed $b,t$, \eqref{eq:matching_ratio} shows that
$(s+1)r_{t+1}(n)/r_t(n)=O_{b,t}(n^{1-d})$.
It remains to prove asymptotic sharpness.

\smallskip
\noindent\emph{Asymptotic sharpness.}
Fix constants $b,t$ and take $n$ sufficiently large. Choose a spanning
path $P$ and a matching $X$ of $t$ hyperedges such that
$|P\cap X|=t-1$. To see that this is possible, take $t-1$
pairwise disjoint edges of $P$, then form the last edge of $X$
from private vertices (vertices belonging to only one path edge)
of $b$ other path edges. Internal path edges
have $b-2\ge1$ private vertices, so there are enough choices.
Let
\[
 J=\{E\in P:E\cap V(X)=\varnothing\},\qquad
 \mathcal H=\bigcup_{E\in J}\T_n^{(b)}[X\cup\{E\}]\ \cup\ \{P\}.
\]
The path has maximum vertex degree two, so
$|J|\ge s-2bt=s-O_{b,t}(1)$; in particular, $|J|\ge2$ for large $n$.
Two members other than $P$ share $X$, and each shares with $P$
the $t-1$ edges of $P\cap X$ and an edge of $J$.
Thus $\mathcal H$ is $t$-intersecting.

It is nontrivial. Any edge common to all members lies in $P$.
For each $E\in P\setminus X$, choose $E'\in J\setminus\{E\}$.
The star at $X\cup\{E'\}$ has $r_{t+1}(n)$ members, of which
at most $r_{t+2}(n)<r_{t+1}(n)$ contain $E$ (and none if the enlarged
edge set is not a forest).
Consequently, the common intersection of $\mathcal H$ is exactly
$P\cap X$, of size $t-1$.
Each star in the union has size $r_{t+1}(n)$, since
$X\cup\{E\}$ is a matching for every $E\in J$.
For distinct $E,E'\in J$, the intersection of the corresponding
stars consists of hypertrees containing $X\cup\{E,E'\}$ and
therefore has size at most $r_{t+2}(n)$.
Moreover, $P$ belongs to none of these stars, since it does not
contain $X$. The first two terms of inclusion--exclusion and
the union bound therefore give
\[
 1+|J|r_{t+1}(n)-\binom{|J|}{2}r_{t+2}(n)
 \le|\mathcal H|\le1+|J|r_{t+1}(n).
\]
By \eqref{eq:matching_ratio},
$|J|r_{t+2}(n)/r_{t+1}(n)=O_{b,t}(n^{2-b})=o(1)$.
Also $r_{t+1}(n)\to\infty$, for example by
\eqref{eq:total} and \eqref{eq:matching_lower}.
Hence $|\mathcal H|=(1-o(1))s r_{t+1}(n)$.
Here $|J|=s-O_{b,t}(1)=\Theta_{b,t}(n)$ is the number of stars
used in the construction. Thus the factor of order $n$ multiplying
$r_{t+1}(n)$ in the bound is necessary for fixed $b,t$.
\end{proof}

\section{Concluding remarks}\label{sec:remarks}
Theorem~\ref{thm:ekr} covers every $1\le t\le s$ once $n\ge N_b$.
Matching stars are extremal precisely in the feasible range
$t\le\lfloor n/b\rfloor$; outside it, the optimal component sizes
are given by \eqref{eq:general_maximum}.
Theorem~\ref{thm:hm} covers
$1\le t\le n/C_b-1$ and therefore allows $t$ to grow linearly
with $n$. There is also a uniform relative bound throughout the full range:
applying Theorem~\ref{thm:nontrivial_cores} with $\St=\F$, $q=s$,
and $\varepsilon=24sB_b/n^d\le1/2$ gives every nontrivial family
\[
 |\F|\le\frac{24sB_b}{n^d}r_t(n)
       =O_b(n^{2-b})r_t(n)\qquad(n\ge N_b).
\]
The finite-$n$ statements of Theorems~\ref{thm:ekr} and~\ref{thm:hm}
do not require $b$ to be fixed as $n$ varies: they hold for every
admissible triple $(n,b,t)$ satisfying the displayed bounds involving
$N_b$ and $C_b$. Thus $b=b(n)$ is allowed whenever those bounds hold.
The asymptotic sharpness assertion is proved with $b$ and $t$ held fixed.

The enumeration, balancing, and degree-decay arguments in
Section~\ref{sec:forests} also hold for ordinary trees, where $b=2$
and $d=1$, with empty products interpreted as one. In particular,
$T_2(F)=n^{m-2}\prod_i f_i$, and the ratio of consecutive matching
stars is $r_{j+1}(n)/r_j(n)=2/n$.
The intersection proofs use $b\ge3$: our spread parameter becomes
$R=n/16$ for $b=2$, which cannot satisfy
$\varepsilon R\ge24s$ with $s=n-1$ and $\varepsilon\le1$.
Moreover, $(s+1)r_{t+1}(n)=2r_t(n)$ whenever $2(t+1)\le n$,
so this bound gives no separation from the largest star.
In the graph case $b=2$, Frankl et al.~\cite{Frankl2025SpanningTrees}
supplement spread approximation with the Lopsided Lov\'asz Local Lemma
and structural reductions. These yield lower bounds on the number of
spanning trees containing a fixed forest while avoiding all edges of a
given tree outside that forest, providing the additional estimates
needed in their intersection argument.
The exceptional intersecting graph family recalled in the introduction
also shows that Theorem~\ref{thm:ekr} cannot extend unchanged to $b=2$.

\bibliographystyle{amsplain}
\bibliography{references}
\end{document}